\documentclass[12pt]{amsart} 
\usepackage{amssymb}
\usepackage{fullpage}
\usepackage[colorlinks=true]{hyperref}
\usepackage{graphicx}

\newcommand{\bbC}{\mathbb{C}}
\newcommand{\bbK}{\mathbb{K}}

\newcommand{\bbQ}{\mathbb{Q}}

\newcommand{\bbZ}{\mathbb{Z}}

\newcommand{\cC}{\mathcal{C}}
\newcommand{\cD}{\mathcal{D}}
\newcommand{\cN}{\mathcal{N}}
\newcommand{\cO}{\mathcal{O}}

\DeclareMathOperator{\core}{sf}

\DeclareMathOperator{\Tr}{Tr}

\numberwithin{equation}{section}

\newtheorem{theorem}{Theorem}[section]
\newtheorem{conjecture}[theorem]{Conjecture}
\newtheorem{corollary}[theorem]{Corollary}
\newtheorem{lemma}[theorem]{Lemma}
\newtheorem{proposition}[theorem]{Proposition}

\theoremstyle{definition}

\newtheorem{remark}[theorem]{Remark}
\newtheorem*{remark-nonum}{Remark}
\newtheorem*{remarks-nonum}{Remarks}

\newcommand{\thmMain}{1.4}

\newcommand{\lemDio}{2.1}
\newcommand{\lemHypg}{2.2}

\newcommand{\eqGDefn}{2.4}

\newcommand{\eqQDefn}{2.7}
\newcommand{\eqKUB}{2.8}
\newcommand{\eqEDefn}{2.9}

\newcommand{\lemUB}{3.7}
\newcommand{\lemGap}{3.8}
\newcommand{\lemTheta}{3.9}

\newcommand{\propRep}{3.1}

\newcommand{\subSectConstr}{2.1} 
\newcommand{\subSectStepiii}{4.4} 

\newcommand{\eqZetaIJRel}{3.23} 
\newcommand{\eqGPNum}{3.24} 
\newcommand{\eqQAbs}{3.25} 
\newcommand{\eqGPb}{3.28}   
\newcommand{\eqOmegaMinusStuffEquality}{3.30} 
\newcommand{\eqOmegaUB}{3.31} 

\newcommand{\eqELB}{4.5} 
\newcommand{\eqQLB}{4.6} 
\newcommand{\eqQUB}{4.7} 
\newcommand{\eqYLUBUsingYK}{4.15} 

\newcommand{\corMainPP}{1.3}
\newcommand{\lemFamilies}{2.4}

\newcommand{\thmMainYLB}{1.3}
\newcommand{\thmMainDLB}{1.4}
\newcommand{\corMainSqr}{1.5}
\newcommand{\lemLBa}{3.2}
\newcommand{\lemGNda}{3.5}
\newcommand{\sectMainProof}{4}

\newcommand{\eqYUBStepOne}{4.16}
\newcommand{\eqYUBStepThree}{4.24}
\newcommand{\eqUBStepFour}{4.27} 
\newcommand{\subsectProofThmForD}{5.6}

\begin{document}

\title[Bounds on the number of squares \ldots]{Bounds on the number of squares in recurrence sequences: $y_{0}=b^{2}$ (II)}

\author{Paul M Voutier}
\address{London, UK}
\email{Paul.Voutier@gmail.com}

\date{}

\begin{abstract}
For binary recurrence sequences, $\left( y_{k} \right)_{k \in \bbZ}$, arising from the solutions
of generalised Pell equations, $X^{2}-dY^{2}=c$, where $y_{0}$ is any
positive square and $c=-2^{\ell}p^{m}$ for an odd prime $p$ and non-negative
integers $\ell$ and $m$, we show that there are at most $4$ distinct squares
with $y_{k}$ sufficiently large. From this result, we also show that there are
at most $7$ distinct squares when $y_{0}=1,2^{2},\ldots,7^{2}$, or once $d$
exceeds an explicit lower bound, without any conditions on the size of such
squares.
\end{abstract}


\maketitle

\section{Introduction}
\label{sect:intro}
\subsection{Background}

Recurrence sequences and the study of their arithmetic properties are an important
part of number theory and have a long history -- see \cite{Everest}. Questions
regarding the squares in
such sequences, in particular, binary recurrence sequences, are important in
their own right, as well as for their connection to Diophantine equations of the
form $aX^{2} - bY^{4} = c$. Such equations are also quartic models of elliptic
curves, adding to their significance.
Recently \cite{V5,V6}, we developed a technique for bounding the number
of distinct squares in binary recurrence sequences. In these papers, we applied
our technique to sequences $\left( y_{k} \right)_{k \in \bbZ}$
with $y_{0}=1$ that arise from the solutions of generalised Pell equations,
$X^{2}-dY^{2}=c$, with $-c$ a positive square \cite{V5} or $1$, $2$, $4$, $8$
or $16$ times
a prime power \cite{V6}. We were able to obtain best possible results for the
sequences in \cite{V6} and for most of the sequences in \cite{V5}.

Since then, we \cite{V7} obtained results when $y_{0}=b^{2}$ for any
positive integer $b$ and $-c$ (corresponding to $-N_{\alpha}$ defined below in
Subsection~\ref{subsect:notation})
is a positive square. Here we obtain further results, extending the
results in \cite{V6} to $y_{0}=b^{2}$ for any positive integer $b$ and
$c=-2^{\ell}p^{m}$ with $p$ an odd prime and $\ell$, $m$ non-negative
integers.

\subsection{Notation}
\label{subsect:notation}

We follow the same notation as in our previous papers on this subject. We recall
that notation here.

Let $a$, $b$ and $d$ be positive integers such that $d$ is not a square.
For $\alpha=a+b^{2}\sqrt{d}$, put $N_{\alpha}
=N_{\bbQ \left( \sqrt{d} \right)/\bbQ}(\alpha)=a^{2}-b^{4}d$ and let
$\varepsilon=\left( t+u\sqrt{d} \right)/2$ be a unit in
$\cO_{\bbQ \left( \sqrt{d} \right)}$ with $t$ and $u$ positive integers.

We define two sequences $\left( x_{k} \right)_{k \in \bbZ}$
and $\left( y_{k} \right)_{k \in \bbZ}$ by
\begin{equation}
\label{eq:yk-defn}
x_{k}+y_{k}\sqrt{d}
=\alpha \varepsilon^{2k}.
\end{equation}

Observe that $x_{0}=a$,
$y_{0}=b^{2}$,
\begin{equation}
\label{eq:yPM1}
y_{1}= \frac{b^{2}\left( t^{2}+du^{2} \right)+2atu}{4}, \hspace*{3,0mm}
y_{-1}=\frac{b^{2}\left( t^{2}+du^{2} \right)-2atu}{4}
\end{equation}
and that both sequences satisfy
the recurrence relation
\begin{equation}
\label{eq:yk-recurrence}
u_{k+1}= \frac{t^{2}+du^{2}}{2} u_{k}-u_{k-1},
\end{equation}
for all $k \in \bbZ$. Here $\left( t^{2}+du^{2} \right)/2
=\Tr_{\bbQ \left( \sqrt{d} \right)/\bbQ} \left( \varepsilon^{2} \right)$.

From \eqref{eq:yk-defn}, we have
\begin{equation}
\label{eq:pell}
x_{k}^{2}-dy_{k}^{2}=N_{\alpha}.
\end{equation}
So this pair of sequences provides a family of solutions of the generalised
Pell equation, $x^{2}-dy^{2}=N_{\alpha}$.

We are interested here in squares in the sequence of $y_{k}$'s, so we set the
coefficient of $\sqrt{d}$ in $\alpha$ to be a square. Furthermore, we choose
$\alpha$ such that $b^{2}$ is the smallest square among the $y_{k}$'s and let
$K$ be the largest negative integer such that $y_{K}>b^{2}$.

\begin{remark}
\label{rem:xk-0}
Before proceeding,
we note that we will require $x_{k} \neq 0$ when $y_{k}$ is a square for our proof
(see, for example,
Lemmas~\ref{lem:quad-rep} and \ref{lem:zeta}). This affects our choice of $\alpha$
here, as if there is a $k$ such that $x_{k}=0$ and $y_{k}$ is a square, then by
\eqref{eq:pell} this $y_{k}$ will be the smallest square in the sequence and
hence the rational part of $\alpha$ will be zero. But this is excluded by our
assumption above that $a$ is a positive integer. This can only happen when
$N_{\alpha}$ is $-d$ times a fourth-power. Since $d$ is not a perfect square,
this case does not affect our work in \cite{V7}. Hence we will assume that
$x_{k} \neq 0$ for all $k$ with $y_{k}$ a square until Section~\ref{sect:xk-0}.
We will treat this outstanding case in Section~\ref{sect:xk-0}.
\end{remark}

For any non-zero integer $n$, let $\core(n)$ be the unique squarefree integer
such that $n/\core(n)$ is a square. We will put $\core(1)=1$.

\subsection{Conjectures}
\label{subsect:conjectures}

We first recall from \cite{V5} some conjectures regarding squares in the sequence
of $y_{k}$'s.

\begin{conjecture}
\label{conj:1-seq}
There are at most four distinct integer squares among the $y_{k}$'s.

If $\core \left( \left| N_{\alpha} \right| \right)|(2p)$ where $p$ is an odd
prime, then there are at most three distinct integer squares among the $y_{k}$'s.

Furthermore, if $\left| N_{\alpha} \right|$ is a perfect square, then there are
at most two distinct integer squares among the $y_{k}$'s.
\end{conjecture}

A more general result than Conjecture~\ref{conj:1-seq} also appears to
be true. Removing the restriction to even powers of $\varepsilon$ in \eqref{eq:yk-defn},
define $\left( x_{k}' \right)_{k \in \bbZ}$ and $\left( y_{k}' \right)_{k \in \bbZ}$
by
\begin{equation}
\label{eq:ykPrime-defn}
x_{k}'+y_{k}'\sqrt{d}
=\alpha \varepsilon^{k}.
\end{equation}

\begin{conjecture}
\label{conj:2-seq}
There are at most four distinct integer squares among the $y_{k}'$'s.

If $\left| N_{\alpha} \right|$ is a prime power or a perfect square,
then there are at most three distinct integer squares among the $y_{k}'$'s.
\end{conjecture}

Computational evidence for these conjectures was presented in Subsection~1.3 of
\cite{V5}.

\subsection{Previous Results}
\label{subsect:prev-results}

The following is Theorem~\thmMain{} in \cite{V5}. It establishes
Conjecture~$\ref{conj:1-seq}$ when $b=1$ and $-N_{\alpha}$ is a square, except
in the special cases noted in parts~(a)--(d) where a slightly weaker result
is proven.

\begin{theorem}
\label{thm:b1-sqr}
Let $b=1$, $a$ and $d$ be positive integers, where $d$ is not a square, $N_{\alpha}<0$
and $-N_{\alpha}$ is a square. There are at most two distinct squares among the $y_{k}$'s,
except in the following cases, where there are at most three distinct squares.

\noindent
{\rm (a)} If $u=1$, $t^{2}-du^{2}=-4$, $N_{\alpha} \equiv 12 \pmod{16}$, $\gcd \left( a^{2},d \right)=1,4$
and one of $y_{\pm 1}$ is a perfect square.

\noindent
{\rm (b)} If $u=1$, $t^{2}-du^{2}=-4$, $64|N_{\alpha}$, $\gcd \left( a^{2},d \right)=4$
and one of $y_{\pm 1}$ is a perfect square.

\noindent
{\rm (c)} If $u=2$, $t^{2}-du^{2}=-4$, $N_{\alpha}$ is odd, $\gcd \left( a^{2},d \right)=1$
and one of $y_{\pm 1}$ is a perfect square.

\noindent
{\rm (d)} If $u=2$, $t^{2}-du^{2}=-4$, $16|N_{\alpha}$, $\gcd \left( a^{2},d \right)=1$
and one of $y_{\pm 1}$ is a perfect square.
\end{theorem}

The following is the main theorem in \cite{V6}.

\begin{theorem}
\label{thm:b1-pp}
Let $b=1$ and let $a$, $m$ and $p$ be non-negative integers with $a \geq 1$ and
$p$ a prime. Conjecture~$\ref{conj:1-seq}$ holds when $N_{\alpha}=-p^{m}$,
$-2p^{m}$, $-4p^{m}$, $-8p^{m}$ and $-16p^{m}$.
\end{theorem}

Lastly, the following is Corollary~\corMainSqr{} in the corrected version of
\cite{V7} (see the arXiv URL in \cite{V7}). It extends Theorem~\ref{thm:b1-sqr}
to any $b$.

\begin{theorem}
\label{thm:any-b-sqr}
Let $a$, $b$ and $d$ be positive integers, where $d$ is not a square
and $-N_{\alpha}$ is a square.

{\rm (a)}  If $b$ is of the form $b=b_{1}b_{2}$ where $b_{1} \in \left\{ 1,5,13,17 \right\}$
and $b_{2}$ has no prime factors congruent to $1 \pmod{4}$, there are at most
five distinct squares among the $y_{k}$'s.

{\rm (b)} If $b$ is not of the form in part~{\rm (a)}, there are at most five
distinct squares among the $y_{k}$'s, provided that
\[
d \geq \frac{17b^{28/13}\left| N_{\alpha} \right|^{1/2}}{u^{24/13}}.
\]
\end{theorem}

\subsection{New Results}
\label{subsect:new-results}

In this paper, we make further progress towards establishing Conjectures~\ref{conj:1-seq}
and \ref{conj:2-seq}. The two theorems below are analogues of Theorems~\thmMainYLB{}
and \thmMainDLB{} in \cite{V7}, respectively.

\begin{theorem}
\label{thm:1.2-seq-new}
Let $a$, $b$ and $d$ be positive integers, where $d$ is not a square, $N_{\alpha}<0$
and $N_{\alpha}=-2^{\ell}p^{m}$ with $p$ an odd prime and $\ell$, $m$ non-negative
integers. There are at most four distinct squares among the
$y_{k}$'s with $k \geq 1$ or $k \leq K$, and
\[
y_{k}> \frac{423b^{8/3}\left| N_{\alpha} \right|^{2}}{\sqrt{d}}.
\]
\end{theorem}

\begin{theorem}
\label{thm:1.3-seq-new}
Let $a$, $b$ and $d$ be positive integers, where $d$ is not a square, $N_{\alpha}<0$
and $N_{\alpha}=-2^{\ell}p^{m}$ with $p$ an odd prime and $\ell$, $m$ non-negative
integers.

{\rm (a)} For $b=1,\ldots, 7$, there are at most four distinct squares among the
$y_{k}$'s with $k \geq 2$ or $k \leq K-1$.

{\rm (b)} For $b \geq 8$, there are at most four distinct squares among the $y_{k}$'s
with $k \geq 2$ or $k \leq K-1$, provided that
\[
d \geq \frac{37b^{28/13} \left| N_{\alpha} \right|^{1/2}}{u^{24/13}}.
\]
\end{theorem}

The following two corollaries are immediate from Theorem~\ref{thm:1.3-seq-new}.

\begin{corollary}
\label{cor:1.3-seq-new}
Let $a$, $b$ and $d$ be positive integers, where $d$ is not a square, $N_{\alpha}<0$
and $N_{\alpha}=-2^{\ell}p^{m}$ with $p$ an odd prime and $\ell$, $m$ non-negative
integers.

{\rm (a)} For $b=1,\ldots,7$, there are at most seven distinct squares among the
$y_{k}$'s.

{\rm (b)} For $b \geq 8$, there are at most seven distinct squares among the $y_{k}$'s,
provided that
\[
d \geq \frac{37b^{28/13} \left| N_{\alpha} \right|^{1/2}}{u^{24/13}}.
\]
\end{corollary}

\begin{corollary}
\label{cor:1.4-seq-new}
Let $a$, $b$ and $d$ be positive integers, where $d$ is not a square, $N_{\varepsilon}=1$,
$N_{\alpha}<0$ and $N_{\alpha}=-2^{\ell}p^{m}$ with $p$ an odd prime and $\ell$,
$m$ non-negative integers.

If the smallest squares among the even-indexed elements of $\left( y_{k}' \right)_{k \in \bbZ}$
and among the odd-indexed elements of $\left( y_{k}' \right)_{k \in \bbZ}$ defined
in \eqref{eq:ykPrime-defn} are both at most $49$, then there are
at most $14$ distinct squares among the $y_{k}'$'s.
\end{corollary}

We require $N_{\varepsilon}=1$ here because, for the odd-indexed elements, we will
apply Corollary~\ref{cor:1.3-seq-new} to the sequences formed from
$\left( \alpha \varepsilon^{2h+1} \right) \varepsilon^{2k}$, where $2h+1$
is the index of the smallest odd-indexed square in the sequence.
But that is only possible if $N \left( \alpha \varepsilon^{2h+1} \right)<0$. Since
$N_{\alpha}<0$, this requires $N_{\varepsilon}=1$.

If the rational part of $\alpha \varepsilon^{2h+1}$ is negative, then the equivalent
of the condition $a>0$ at the start of Subsection~\ref{subsect:notation} does not hold.
In this case, replace $\alpha\varepsilon^{2h+1}$ by its negative conjugate. This
has positive rational part and its only effect on the $y$-values is to change the
sign of their index.

If the rational part of $\alpha\varepsilon^{2h+1}$ is zero, then Section~\ref{sect:xk-0}
shows that there are at most three distinct squares among the odd-indexed elements.
Otherwise, after replacing $\alpha\varepsilon^{2h+1}$ by its negative conjugate
if necessary, Corollary~\ref{cor:1.3-seq-new}(a) applies and gives at most seven
such squares.

As in \cite{V6}, we can use Lemma~\lemFamilies{} in \cite{V6} to establish generalisations
of Corollary~\corMainPP{} there, upper bounds on the number of coprime positive
integer solutions of $X^{2}-dY^{4}=N_{\alpha}$, where $N_{\alpha}=-2^{\ell}p^{m}$.

The code used in this work is publicly available at
\url{https://github.com/PV-314/hypgeom/any-b/Na-PP}. The author is very
happy to help interested readers who have any questions, problems or
suggestions for the use of this code.

\section{Diophantine Approximation via Hypergeometric Functions}
\label{sect:dio}

For completeness and ease for the reader, we collect here the notation and
statements from Section~2 of \cite{V7} that we will need in our proof.

In the published version of this work, we had Lemma~\lemDio{} from \cite{V7}.
However, as we state in the next subsection, this lemma is not always applicable
here. Instead we need Lemma~\lemDio{} from \cite{V8}, which we quote here.

\begin{lemma}
\label{lem:2.1}
Let $\theta \in \bbC$ and let $\bbK$ be an imaginary quadratic field. Suppose that
there exist $k_{0},\ell_{0} > 0$ and $E,Q > 1$ such that for all non-negative
integers $r$, there are algebraic integers $p_{r}$ and $q_{r}$ in $\bbK(g')$,
where $0 \neq {g'}^{2} \in \cO_{\bbK}$ with ${g'}^{r}p_{r}, {g'}^{r}q_{r} \in \cO_{\bbK}$,
$\left| q_{r} \right| < k_{0}Q^{r}$ and
$\left| q_{r} \theta - p_{r} \right| \leq \ell_{0}E^{-r}$ satisfying 
$p_{r}q_{r+1} \neq p_{r+1}q_{r}$.

For any algebraic integers $p$ and $q$ in $\bbK$, not both zero, let $r_{0}$ be the smallest
positive integer such that $\left( Q-1/E  \right)\ell_{0}|q|/\left( Q-1 \right)<cE^{r_{0}}$,
where $0<c<1$.

\noindent
{\rm (a)} We have 
\[
\left| q\theta - p \right| 
> \frac{1-c/E}{k_{0}Q^{r_{0}+1}}.
\]

\noindent
{\rm (b)} When $pq_{r_{0}} \neq qp_{r_{0}}$, we have
\[
\left| q\theta - p \right| 
> \frac{1-c}{k_{0}Q^{r_{0}}}.
\]
\end{lemma}

\subsection{Construction of Approximations}
\label{subsect:const}

Let $t'$, $u_{1}$ and $u_{2}$ be rational integers with $t'<0$ and $u_{2} \neq 0$
such that $U=\left( u_{1}+u_{2}\sqrt{t'} \right)/2$ is an algebraic integer in
$\bbK=\bbQ \left( \sqrt{t'} \right)$ with $\sigma(U)=\left( u_{1}-u_{2}\sqrt{t'} \right)/2$
as its algebraic (and complex) conjugate. Put $\omega = U/\sigma(U)$ and write
$\omega=e^{i\varphi}$, where $-\pi<\varphi \leq \pi$. For any real number
$\nu$, we shall put $\omega^{\nu}= e^{i\nu\varphi}$ -- unless otherwise stated,
we will use this convention throughout this paper.

Suppose that $\alpha$, $\beta$ and $\gamma$ are complex numbers and $\gamma$ is
not a non-positive integer. We denote by ${}_{2}F_{1}(\alpha, \beta; \gamma; z)$
the classical (or Gauss) hypergeometric function of the complex variable $z$.
For integers $m$ and $n$ with $0 < m < n/2$, $\gcd(m,n) = 1$ and $r$ a non-negative
integer, put $\nu=m/n$ and
\[
X_{m,n,r}(z)={}_{2}F_{1}(-r-\nu, -r; 1-\nu; z), \quad
Y_{m,n,r}(z)=z^{r}X_{m,n,r} \left(z^{-1} \right)
\]
and
\[
R_{m,n,r}(z)
= (z-1)^{2r+1} \frac{\nu \cdots (r+\nu)}{(r+1) \cdots (2r+1)} 
		   {} _{2}F_{1} \left( r+1-\nu, r+1; 2r+2; 1-z \right).
\]

We let $D_{n,r}$ denote the smallest positive integer such that $D_{n,r} X_{m,n,r}(x) \in \bbZ[x]$
for all $m$ as above. For $d' \in \bbZ$ with $d' \neq 0$, we define $N_{d',n,r}$ to be
the largest integer such that $\left( D_{n,r}/ N_{d',n,r} \right)X_{m,n,r}\left( 1-\sqrt{d'}\,x \right)
\in \bbZ \left[ \sqrt{\core \left( d' \right)} \right] [x]$, again for all $m$ as above. We will use
$v_{p}(x)$ to denote the exponent of the largest power of a prime $p$ which divides
the rational number $x$. We put
\begin{equation}
\label{eq:ndn-defn}
\cN_{d',n} =\prod_{p|n} p^{\min(v_{p}(d')/2, v_{p}(n)+1/(p-1))}.
\end{equation}

In what follows, we shall restrict our attention to $m=1$ and $n=4$, so $\nu=1/4$.

As in equation~(\eqGDefn{}) in \cite{V5}, put
\begin{align}
\label{eq:g-defn}
g_{1}  & = \gcd \left( u_{1}, u_{2} \right), \\
g_{2}  & = \gcd \left( u_{1}/g_{1}, t' \right), \nonumber \\
g_{3}  & = \left\{
             \begin{array}{ll}
	             1 & \text{if $t' \equiv 1 \pmod{4}$ and $\left( u_{1}-u_{2} \right)/g_{1} \equiv 0 \pmod{2}$}, \\
	             2 & \text{if $t' \equiv 3 \pmod{4}$ and $\left( u_{1}-u_{2} \right)/g_{1} \equiv 0 \pmod{2}$},\\
	             4 & \text{otherwise,}
             \end{array}
             \right. \nonumber \\
g      & = g_{1}\sqrt{g_{2}/g_{3}}. \nonumber
\end{align}

Then we can put
\begin{align}
\label{eq:7}
p_{r} &= \frac{D_{4,r}}{N_{d',4,r}} \left( \frac{u_{1}-u_{2}\sqrt{t'}}{2g} \right)^{r} X_{1,4,r}(\omega),\\
q_{r} &= \frac{D_{4,r}}{N_{d',4,r}} \left( \frac{u_{1}-u_{2}\sqrt{t'}}{2g} \right)^{r} Y_{1,4,r}(\omega)
\hspace*{3.0mm} \text{ and} \nonumber \\
R_{r} &= \frac{D_{4,r}}{N_{d',4,r}} \left( \frac{u_{1}-u_{2}\sqrt{t'}}{2g} \right)^{r} R_{1,4,r}(\omega),
\nonumber
\end{align}
where
\begin{equation}
\label{eq:d-defn}
d'=\left( U-\sigma(U) \right)^{2}/g^{2} = u_{2}^{2}t'/g^{2}.
\end{equation}

As in Subsection~\subSectConstr{}\, of \cite{V5} (as corrected in the arXiv link
in \cite{V5}), $p_{r}$ and $q_{r}$ are algebraic integers. For $r$ even, they
are in $\bbQ \left( \sqrt{t'} \right)$. For $r$ odd, they are in
$\bbQ \left( \sqrt{t'}, \sqrt{\core \left( g_{2}g_{3} \right)} \right)$ with
$\sqrt{\core \left( g_{2}g_{3} \right)} \, p_{r}$ and $\sqrt{\core \left( g_{2}g_{3} \right)} \, q_{r}$
in $\bbQ \left( \sqrt{t'} \right)$.
So we will apply Lemma~\ref{lem:2.1} with $g'=\sqrt{\core \left( g_{2}g_{3} \right)}$.
Also
\[
q_{r}\omega^{1/4} - p_{r}=R_{r}.
\]

As in equations~(\eqQDefn{}), (\eqKUB{}), (\eqEDefn{}) in \cite{V5} and the
expression for $\ell_{0}$ immediately after them there, we have
\begin{align}
\label{eq:q-defn}
Q &= \frac{\cD_{4} \left| \left| u_{1} \right| + \sqrt{u_{1}^{2}-t'u_{2}^{2}} \right|}{|g|\cN_{d',4}},\\
\label{eq:k-UB}
k_{0} &<1.257, \\
\label{eq:e-defn}
E &= \frac{|g|\cN_{d',4} \left|  \left| u_{1} \right|  + \sqrt{u_{1}^{2}-t'u_{2}^{2}} \right|}{\cD_{4}u_{2}^{2}|t'|}, \\
\label{eq:ell0-defn}
\ell_{0} &= \cC_{4,2}|\varphi|=0.2685|\varphi|, \text{ where} \\
\label{eq:d4-defn}
\cD_{4} &= e^{1.68}.
\end{align}

\section{Lemmas about $\left( x_{k} \right)_{k \in \bbZ}$ and $\left( y_{k} \right)_{k \in \bbZ}$}
\label{sect:prelim}

We start by stating Proposition~\propRep{}(c) of \cite{V5} when $\epsilon$ is
a unit in $\bbQ \left( \sqrt{d} \right)$. This is the result that permits us to
use the hypergeometric method.

\begin{lemma}
\label{lem:quad-rep}
Let $a \neq 0$, $b>0$ and $d$ be rational integers such that $d$ is not a square.
Put $\alpha=a+b^{2} \sqrt{d}$ and denote $N_{\bbQ(\sqrt{d})/\bbQ}(\alpha)$
by $N_{\alpha}$. Suppose that $N_{\alpha}$ is not a square, $x \neq 0$ and $y>0$
are rational integers with
\begin{equation}
\label{eq:quad-rep-assumption}
x+y^{2} \sqrt{d} = \alpha \epsilon^{2},
\end{equation}
where $\epsilon=\left( t+u\sqrt{d} \right)/2 \in \cO_{\bbQ \left( \sqrt{d} \right)}$
with $t$ and $u$ non-zero rational integers, norm $N_{\epsilon}=\pm 1$.

If $\core \left( \left| N_{\alpha} \right| \right) = 2^{\ell}p^{m}$ where $\ell, m \in \{ 0, 1 \}$
with $\ell+m \geq 1$ and $p$ is an odd prime, then we can write
\begin{align}
f^{2} \left( x + N_{\epsilon}\sqrt{N_{\alpha}} \right)
&= \left( a+\sqrt{N_{\alpha}} \right) \left( r + s\sqrt{\core \left( N_{\alpha} \right)} \right)^{4}
\quad \text{and} \nonumber \\
\label{eq:fy-rel}
fy &= b \left( r^{2}-\core \left( N_{\alpha} \right)s^{2} \right),
\end{align}
for some integers $f$, $r$ and $s$ satisfying $f \neq 0$,
$f \mid \left( 4b^{2} \right)$ when $N_{\alpha} \equiv 1 \pmod{4}$
and $4|d$, and $f \mid \left( 2b^{2} \right)$ otherwise.
\end{lemma}

We will also need lower bounds for the elements in our sequences.
Recall that $K$ is the largest negative integer such that $y_{K}>b^{2}$.

\begin{lemma}
\label{lem:Y-LB1}
Let the $y_{k}$'s be defined by $\eqref{eq:yk-defn}$ with the notation and assumptions
there. Suppose that $N_{\alpha}<0$.
 
\noindent
{\rm (a)}
For all $k$, $2y_{k}$ is a positive integer.
The sequences $\left( y_{k} \right)_{k \geq 0}$ and 
$\left( y_{K+1}, y_{K}, y_{K-1}, y_{K-2}, \ldots \right)$ are increasing sequences of positive numbers.

\noindent
{\rm (b)}
 For $N_{\varepsilon}=1$ or $du^{2} \geq 300$, we have
\begin{equation}
\label{eq:yLB-gen}
y_{k} \geq
\left\{
\begin{array}{ll}
\left( \left| N_{\alpha} \right|u^{2} / \left( 4b^{2} \right) \right) \left( 0.99du^{2} \right)^{k-1} & \text{for $k>0$,} \\
\left( \left| N_{\alpha} \right|u^{2} / \left( 4b^{2} \right) \right) \left( 0.99du^{2} \right)^{\max(0,K-k)} & \text{for $k<0$.}
\end{array}
\right.
\end{equation}

For any $d>1$, we can replace $0.99$ by $0.4$.
\end{lemma}

\begin{proof}
These are parts~(a) and (c) of Lemma~\lemLBa{} of \cite{V7}.
\end{proof}

Next we state a gap principle separating distinct squares in the sequence of
$y_{k}$'s defined by \eqref{eq:yk-defn}.

For any $k \neq 0$ with $y_{k}=y^{2}$ and recalling our assumption that
$x_{k} \neq 0$ when $y_{k}$ is a square, we let $r_{k}$ and
$s_{k}$ be the quantities in Lemma~\ref{lem:quad-rep}, using $\left( x_{k}, y_{k} \right)$
for $\left( x,y^{2} \right)$ there, let $\omega= \left( a-\sqrt{N_{\alpha}} \right)
/ \left( a+\sqrt{N_{\alpha}} \right)$ and let $\zeta_{4}^{(k)}$ be the $4$th
root of unity such that
\begin{equation}
\label{eq:zetaj-defn}
\left| \omega^{1/4} - \zeta_{4}^{(k)} \frac{r_{k}+s_{k}\sqrt{\core \left( N_{\alpha} \right)}}
{r_{k}-s_{k}\sqrt{\core \left( N_{\alpha} \right)}} \right|
= \min_{0 \leq \ell \leq 3} \left| \omega^{1/4} - e^{2\ell \pi i/4} \frac{r_{k}+s_{k}\sqrt{\core \left( N_{\alpha} \right)}}{r_{k}-s_{k}\sqrt{\core \left( N_{\alpha} \right)}} \right|.
\end{equation}

\begin{lemma}
\label{lem:gap}
Let the $y_{k}$'s be defined as in $\eqref{eq:yk-defn}$ with $N_{\alpha}<0$.
Suppose that $N_{\alpha}=-2^{\ell}p^{m}$, where $p$ is an odd prime, $\ell$ and
$m$ are non-negative integers and that $N_{\alpha}$ is not $-1$ times a square.
If $y_{k_{1}}$, $y_{k_{2}}$ and $y_{k_{3}}$
are three distinct squares with $k_{1}, k_{2}, k_{3} \neq 0$ and
$y_{k_{3}}>y_{k_{2}}>y_{k_{1}} \geq \max \left( 4\sqrt{\left| N_{\alpha} \right|/d},
4b^{2}\left| N_{\alpha} \right|/d \right)$,
then there exist distinct $i,j \in \{ k_{1},k_{2}, k_{3} \}$ such that
\[
y_{j} > \frac{15d^{2}}{b^{4}\left| N_{\alpha} \right|^{2}} y_{i}^{3}.
\]
Moreover, we can let $i$ and $j$ be any distinct elements of $\{ k_{1},k_{2}, k_{3} \}$
such that $y_{j}>y_{i}$ and $\zeta_{4}^{(i)}= \pm \zeta_{4}^{(j)}$, where
$\zeta_{4}^{(i)}$ and $\zeta_{4}^{(j)}$ are as defined in $\eqref{eq:zetaj-defn}$.
\end{lemma}

\begin{proof}
Our proof here will follow closely the proof of Lemma~\lemGap(b) of \cite{V5}.
As in the proof of Lemma~\lemGap{}(b) of \cite{V5} (just above equation~(3.19) on page~314),
we put $\omega= \left( a-\sqrt{N_{\alpha}} \right) / \left( a+\sqrt{N_{\alpha}} \right)$.
We can write $\omega=e^{\varphi \sqrt{-1}}$, where $-\pi<\varphi \leq \pi$.
For any real number $\nu$, we shall put $\omega^{\nu}=e^{\nu \varphi \sqrt{-1}}$.

For any $j$ among $k_{1}$, $k_{2}$ and $k_{3}$, let $\zeta_{4}^{(j)}$ be the
$4$th root of unity such that
\[
\left| \omega^{1/4} - \zeta_{4}^{(j)} \frac{r_{j}+s_{j}\sqrt{\core \left( N_{\alpha} \right)}}{r_{j}-s_{j}\sqrt{\core \left( N_{\alpha} \right)}} \right|
= \min_{0 \leq k \leq 3} \left| \omega^{1/4} - e^{(2k \pi/4) \sqrt{-1}} \, \frac{r_{j}+s_{j}\sqrt{\core \left( N_{\alpha} \right)}}{r_{j}-s_{j}\sqrt{\core \left( N_{\alpha} \right)}} \right|,
\]
where $r_{j}$ and $s_{j}$ are the values of $r$ and $s$ in Lemma~\ref{lem:quad-rep}
associated with $x=x_{j}$ and $y^{2}=y_{j}$, letting $\alpha$ in Lemma~\ref{lem:quad-rep}
be the same $\alpha=a+b^{2}\sqrt{d}$ that we use throughout this paper. As argued
in Remark~\ref{rem:xk-0} and shown in Section~\ref{sect:xk-0}, we can assume that
we have $a \neq 0$ as required for Lemma~\ref{lem:quad-rep}. Similarly, we let
$f_{j}$ be the value of $f$ in Lemma~\ref{lem:quad-rep} associated with $y_{j}$.
From Lemma~\ref{lem:quad-rep}, we have $f_{j}| \left( 4b^{2} \right)$.

Choose $i$ and $j$ from among $k_{1}$, $k_{2}$ and $k_{3}$ so that $y_{j}>y_{i}$
and $\zeta_{4}^{(i)}= \pm \zeta_{4}^{(j)}$.

Suppose $f_{i}f_{j} \leq 4b^{4}$. From Lemma~\lemGap(b) of \cite{V5}, we have
\[
y_{j}>15.36 \left( \frac{d}{4b^{2} \left| N_{\alpha} \right|} \right)^{2} y_{i}^{3}.
\]

Since $y_{i} \geq 4b^{2}\left| N_{\alpha} \right|/d$, we obtain $y_{j}>15.36y_{i}$.
Equation~(\eqGPb) of \cite{V5} states that
\begin{equation}
\label{eq:gpEqnb1}
\frac{2b}{\sqrt{f_{i}f_{j}} \left( y_{i}y_{j} \right)^{1/4}}
< 0.5051b^{2} \sqrt{\frac{\left| N_{\alpha} \right|}{a^{2}+\left| N_{\alpha} \right|}}
\left( \frac{1}{y_{i}} + \frac{1}{y_{j}} \right)
= 0.5051 \sqrt{\frac{\left| N_{\alpha} \right|}{d}}
\left( \frac{1}{y_{i}} + \frac{1}{y_{j}} \right).
\end{equation}

Applying $f_{i}f_{j} \leq 4b^{4}$ to this inequality and $y_{j}>15.36y_{i}$ to
the right-hand side of it, it follows that
\begin{equation}
\label{eq:yj-LB2}
y_{j}>11.93 \left( \frac{d}{b^{2}\left| N_{\alpha} \right|} \right)^{2} y_{i}^{3}.
\end{equation}

From \eqref{eq:yj-LB2} and $y_{i} \geq 4b^{2}\left| N_{\alpha} \right|/d$, we
have $y_{j}>190.88y_{i}$. Applying this to \eqref{eq:gpEqnb1}, as we did above
with $y_{j}>15.36y_{i}$, we obtain
\[
y_{j}>15.04 \left( \frac{d}{b^{2}\left| N_{\alpha} \right|} \right)^{2} y_{i}^{3}.
\]

Now suppose that $f_{i}f_{j}>4b^{4}$. We show that we are able to improve
\eqref{eq:gpEqnb1} here. In this case, both $f_{i}$ and $f_{j}$ are divisors
of $4b^{2}$. Since $f_{i}f_{j}>4b^{4}$ and $f_{i}, f_{j} \leq 4b^{2}$, it
follows that $f_{i}, f_{j}>b^{2}$. Hence both $4b^{2}/f_{i}$ and $4b^{2}/f_{j}$
are positive integers strictly less than $4$. Thus
$f_{i},f_{j} \in \left\{ 4b^{2}/3, 2b^{2}, 4b^{2} \right\}$. Note that the first
value requires $3|b$. Also, at least one of $f_{i}, f_{j}$ must equal $4b^{2}$
(otherwise, their product would be at most $4b^{4}$). Thus, the possibilities
for the pairs, up to interchanging $f_{i}$ and $f_{j}$, are
\[
\left( f_{i}, f_{j} \right)
\in \left\{ \left( 4b^{2}, 4b^{2} \right), \left( 4b^{2}, 2b^{2} \right),
\left( 4b^{2}, 4b^{2}/3 \right) \right\}.
\]

Since one of the values is $4b^{2}$, from Lemma~\ref{lem:quad-rep}, we must have
$N_{\alpha} \equiv 1 \pmod{4}$ and $4|d$ (i.e., $a$ is odd).

In equation~\eqref{eq:fy-rel} in Lemma~\ref{lem:quad-rep}, we have
\[
f_{i} \sqrt{y_{i}} = b \left( r_{i}^{2}-\core \left( N_{\alpha} \right) s_{i}^{2} \right),
\]
and similarly for $f_{j} \sqrt{y_{j}}$. Since both $f_{i}/b$ and $f_{j}/b$ are even
integers, it follows that both $r_{i}^{2}-\core \left( N_{\alpha} \right) s_{i}^{2}$
and $r_{j}^{2}-\core \left( N_{\alpha} \right) s_{j}^{2}$ are even.
Since $N_{\alpha}$ is odd, we see that $r_{i}$ and $s_{i}$, as well as $r_{j}$
and $s_{j}$, must have the same parity. Letting $x+y\sqrt{\core \left( N_{\alpha} \right)}$
be as in the proof of Lemma~\lemGap{} of \cite{V5} (see immediately above
equation~(\eqZetaIJRel{}) there), we write
\begin{align*}
x+y\sqrt{\core \left( N_{\alpha} \right)}
&= \left( r_{i}-s_{i}\sqrt{\core \left( N_{\alpha} \right)} \right)
\left( r_{j}+s_{j}\sqrt{\core \left( N_{\alpha} \right)} \right) \\
&=r_{i}r_{j}-s_{i}s_{j}\core \left( N_{\alpha} \right) + \left( r_{i}s_{j}-r_{j}s_{i} \right) \sqrt{\core \left( N_{\alpha} \right)}.
\end{align*}

Since $r_{i}$ and $s_{i}$ have the same parity, as do $r_{j}$ and $s_{j}$, and
since $N_{\alpha}$ is odd, we see that $x$ and $y$ are both even here.
Therefore, the expressions for the numerators in equation~(\eqGPNum{}) in \cite{V5}
have absolute values at least $4$ when $\zeta_{4}^{(j)}=\pm \zeta_{4}^{(i)}$ and
so
\begin{equation}
\label{eq:gpEqnb2}
\frac{4b}{\sqrt{f_{i}f_{j}} \left( y_{i}y_{j} \right)^{1/4}}
< 0.5051 \sqrt{\frac{\left| N_{\alpha} \right|}{d}}
\left( \frac{1}{y_{i}} + \frac{1}{y_{j}} \right).
\end{equation}

The worst case is $f_{i}=f_{j}=4b^{2}$. Here, since $y_{j}>y_{i}$, the right-hand
side of \eqref{eq:gpEqnb2} is less than $0.5051\sqrt{\left| N_{\alpha} \right|/d} \left( 2/y_{i} \right)$,
and we obtain $y_{j}>0.96 \left( d/\left( b^{2} \left| N_{\alpha} \right| \right) \right)^{2}y_{i}^{3}$
(in the case when one of $f_{i}$ is $2b^{2}$ and the other is $4b^{2}$, then we
obtain
$y_{j}>3.84 \left( d/\left( b^{2} \left| N_{\alpha} \right| \right) \right)^{2}y_{i}^{3}$).
Since $y_{i} \geq 4b^{2} \left| N_{\alpha} \right|/d$, this yields $y_{j}>15.36y_{i}$.

Using the same process as for $f_{i}f_{j} \leq 4b^{4}$ (except that we need
three passes through that process, rather than two passes as we used for
$f_{i}f_{j} \leq 4b^{4}$), we obtain our result. The coefficients obtained in
these three passes are $11.93$, $15.04$ and $15.11\ldots$.
\end{proof}

Now we collect some results that we will need for bounding quantities that
arise in the proof of the main result in the following section (Section~\ref{sect:proof-thm12}).

\begin{lemma}
\label{lem:zeta}
Let the $x_{k}$'s and $y_{k}$'s be defined as in $\eqref{eq:yk-defn}$ with the
notation and assumptions there, including $x_{k} \neq 0$. Suppose that $N_{\alpha}<0$.

{\rm (a)} Let $y_{k} \geq 4\sqrt{\left| N_{\alpha} \right|/d}$ be a square and put
\[
\omega_{k}=\left( x_{k} + N_{\varepsilon^{k}}\sqrt{N_{\alpha}} \right)
/\left( x_{k} - N_{\varepsilon^{k}}\sqrt{N_{\alpha}} \right)
=e^{i\varphi_{k}}
\]
with $-\pi<\varphi_{k} \leq \pi$.
Then
\[
\left| \varphi_{k} \right| <\frac{2.29\sqrt{\left| N_{\alpha} \right|}}{\left| x_{k} \right|}
<0.6.
\]

{\rm (b)} Let $y_{k}$ and $y_{\ell}$ be two squares with $k,\ell \neq 0$
\, \footnote{The condition $k, \ell \neq 0$ was left out of the statement of
Lemma~3.9(b) in the published version of \cite{V5}, but is required both there
and here.},
$y_{\ell}>y_{k} \geq 4\sqrt{\left| N_{\alpha} \right|/d}$ and
$\omega_{k}$ be as in part~{\rm(a)} and put
\[
x+y\sqrt{\core \left( N_{\alpha} \right)}
=\left(  r_{k}-s_{k}\sqrt{\core \left( N_{\alpha} \right)} \right)
\left( r_{\ell}+s_{\ell}\sqrt{\core \left( N_{\alpha} \right)} \right)
\]
with $r_{k},r_{\ell},s_{k},s_{\ell}$ as in Lemma~$\ref{lem:quad-rep}$.
Suppose that the quantities $\zeta_{4}^{(k)}$ and $\zeta_{4}^{(\ell)}$
defined in \eqref{eq:zetaj-defn} satisfy $\zeta_{4}^{(k)}= \pm \zeta_{4}^{(\ell)}$.
Then
\[
\min_{0 \leq j \leq 3} \left| \omega_{k}^{1/4} - \zeta_{4}^{j}
\frac{x-y\sqrt{\core \left( N_{\alpha} \right)}}{x+y\sqrt{\core \left( N_{\alpha} \right)}} \right|
\]
occurs for either $j=0$ or $j=2$, where $\zeta_{4}$ is a primitive $4$th root of
unity.
\end{lemma}

\begin{proof}
This is Lemma~\lemTheta{} of \cite{V5}.
\end{proof}

\section{Proof of Theorem~\ref{thm:1.2-seq-new}}
\label{sect:proof-thm12}

First note that by Theorem~1.3 of \cite{V7}, we may assume that $-N_{\alpha}$
is not a square. Thus we may apply Lemmas~\ref{lem:quad-rep} and \ref{lem:gap}
as stated above.

As in \cite{V5,V6,V7}, we break the proof into several parts. As with
Section~\sectMainProof{}\, of \cite{V7}, we provide Maple code to check the
algebraic manipulations that follow. This code can be found in the file
\verb!bAll-nAPP-proof-calcs.txt! in the \verb!maple! subdirectory of the
GitHub URL provided at the end of
Section~\ref{sect:intro}.

\subsection{Choice of indices}
\label{subsect:k-and-ell}

To prove Theorem~\ref{thm:1.2-seq-new}, we assume there are five distinct squares,
\begin{equation}
\label{eq:ass-7aa}
y_{k_{5}}>y_{k_{4}}>y_{k_{3}}>y_{k_{2}}>y_{k_{1}}
\geq \max \left( 4\sqrt{\left| N_{\alpha} \right|/d}, 4b^{2}\left| N_{\alpha} \right|/d \right),
\end{equation}
satisfying
\begin{equation}
\label{eq:ass-5aa-xxx}
k_{1} \geq 1 \text{ or } k_{1} \leq K,
\end{equation}
recalling that $K$ is the largest negative integer such that $y_{K}>b^{2}$ and
obtain a contradiction for $y_{k_{1}}$ sufficiently large. Also recall that by
our choice of $\alpha$, there are no squares among the $y_{k}$ with $K+1 \leq k \leq 0$
distinct from $y_{0}$. The last inequality in \eqref{eq:ass-7aa} arises from our
gap principle in Lemma~\ref{lem:gap} and from Lemma~\ref{lem:zeta}.
The condition in \eqref{eq:ass-5aa-xxx} also ensures that the conditions to apply
both our gap principle lemma and Lemma~\ref{lem:quad-rep} are met.

We will use extensively the notation from Section~\ref{sect:dio} in both this
section and the following section.

Using \eqref{eq:zetaj-defn}, we have $\zeta_{4}^{(k_{1})}$, \ldots, $\zeta_{4}^{(k_{5})}$
associated with $k_{1}$,\ldots, $k_{5}$. At least three of them belong to
$\{ -1, 1 \}$ or at least three of them belong to $\{ -i, i \}$. We label the
three associated indices as $m_{1}$, $m_{2}$ and $m_{3}$, where
$y_{m_{1}}<y_{m_{2}}<y_{m_{3}}$. If the set of at least three elements contains more than
three elements, we can choose any three among them.

It is in order to have three such elements that we assume we have five squares in
\eqref{eq:ass-7aa}. The reason we want three such elements is because in
Subsections~\ref{subsect:thm12-step-i} and \ref{subsect:thm12-step-ii}, we will
apply the gap principle in Lemma~\ref{lem:gap} twice.

As in Lemma~\ref{lem:zeta}, we put
$\omega_{m_{1}}=\left( x_{m_{1}}+N_{\varepsilon^{m_{1}}}\sqrt{N_{\alpha}} \right)
/\left( x_{m_{1}}-N_{\varepsilon^{m_{1}}}\sqrt{N_{\alpha}} \right)$
and let $\eta$ be the $4$th root of unity such that
\[
\left| \omega_{m_{1}}^{1/4} - \eta \frac{x-y\sqrt{\core \left( N_{\alpha} \right)}}{x+y\sqrt{\core \left( N_{\alpha} \right)}} \right|
= \min_{0 \leq j \leq 3} \left| \omega_{m_{1}}^{1/4} - e^{2j\pi i/4} \frac{x-y\sqrt{\core \left( N_{\alpha} \right)}}{x+y\sqrt{\core \left( N_{\alpha} \right)}} \right|,
\]
where $x+y\sqrt{\core \left( N_{\alpha} \right)}=\left(  r_{m_{1}}-s_{m_{1}}\sqrt{\core \left( N_{\alpha} \right)} \right)
\left( r_{m_{3}}+s_{m_{3}}\sqrt{\core \left( N_{\alpha} \right)} \right)$ with
$\left( r_{m_{1}}, s_{m_{1}} \right)$ and $\left( r_{m_{3}}, s_{m_{3}} \right)$
as in Lemma~\ref{lem:quad-rep},
which are associated with $\left( x_{m_{1}}, y_{m_{1}} \right)$ and
$\left( x_{m_{3}}, y_{m_{3}} \right)$, respectively.
We can take $\eta \in \bbQ \left( \sqrt{\core \left( N_{\alpha} \right)} \right)$.
Applying Lemma~\ref{lem:zeta}(b) with $k=m_{1}$ and $\ell=m_{3}$,
along with the fact that we chose $m_{1}$, $m_{2}$ and $m_{3}$ so that
$\zeta^{(m_{1})}=\pm \zeta^{(m_{3})}$, we find that $\eta=\pm 1 \in \bbQ \left( \sqrt{\core \left( N_{\alpha} \right)} \right)$.

This is important for us here as $\eta \left( x-y\sqrt{\core \left( N_{\alpha} \right)} \right)
/ \left( x+y\sqrt{\core \left( N_{\alpha} \right)} \right)$ must be in an
imaginary quadratic field in order to apply Lemma~\ref{lem:2.1} to obtain a
lower bound for the rightmost quantity in \eqref{eq:29} below.

\subsection{Further Prerequisites}
\label{subsect:preq}

In this subsection, we collect some inequalities that will be required in the
subsections that follow.

From (\eqOmegaUB{}) in the proof of Lemma~\lemTheta{}(b) in \cite{V5}, we have
\[
\left| \omega_{m_{1}}^{1/4} - \eta \frac{x-y\sqrt{\core\left( N_{\alpha} \right)}}{x+y\sqrt{\core\left( N_{\alpha} \right)}} \right|
<0.127.
\]
This is where we need \eqref{eq:ass-7aa} and
$m_{1}, m_{3} \neq 0$, which follows from \eqref{eq:ass-5aa-xxx}.

Thus we can apply Lemma~\lemUB{}(a) in \cite{V5} with $c_{1}=0.127$ to find that
\begin{equation}
\label{eq:29}
\frac{2\sqrt{\left| N_{\alpha} \right|}}{\sqrt{d} \, y_{m_{3}}}
= \left| \omega_{m_{1}} - \left( \frac{x-y\sqrt{\core \left( N_{\alpha} \right)}}{x+y\sqrt{\core \left( N_{\alpha} \right)}} \right)^{4} \right|
> 3.959 \left| \omega_{m_{1}}^{1/4} - \eta \frac{x-y\sqrt{\core \left( N_{\alpha} \right)}}{x+y\sqrt{\core \left( N_{\alpha} \right)}} \right|.
\end{equation}

The equality on the left-hand side is from the equalities in (\eqOmegaMinusStuffEquality{})
in the proof of Lemma~\lemTheta{}(b) in \cite{V5}.

We need to derive a lower bound for the far-right quantity in \eqref{eq:29}.
To do so, we shall use the lower bounds in Lemma~\ref{lem:2.1} with
a sequence of good approximations $p_{r}/q_{r}$ obtained from the hypergeometric
functions. We use the notation of Subsection~\ref{subsect:const} to define these
good approximations. Let $t'=\core \left( N_{\alpha} \right)$, $u_{1}=2x_{m_{1}}$
and $u_{2}=2N_{\varepsilon^{m_{1}}} \sqrt{N_{\alpha}/\core \left( N_{\alpha} \right)}$.
So $\omega$ in that subsection is our $\omega_{m_{1}}$ here. $D_{4,r}$ and
$N_{d',4,r}$ are as defined just before equation~\eqref{eq:ndn-defn}. Recall that
$g$ and $d'$ are as defined in equations~\eqref{eq:g-defn} and \eqref{eq:d-defn},
respectively. We can now define $p_{r}$ and $q_{r}$ as in \eqref{eq:7}.

We now collect inequalities for the quantities required for Lemma~\ref{lem:2.1}.

To obtain our lower bounds for $E$ and $Q$ in \eqref{eq:E-LB1} and \eqref{eq:Q-LB1}
below (where $E$ and $Q$ are defined in \eqref{eq:e-defn} and \eqref{eq:q-defn}
above), we also need
\begin{equation}
\label{eq:ass-3}
y_{m_{1}} \geq \max \left( 4\left| N_{\alpha} \right|/\sqrt{d}, |g|\cN_{d',4}/\sqrt{d} \right).
\end{equation}

In Lemma~\lemGNda{} of \cite{V7}, we showed that $|g|\cN_{d',4} \leq 16\sqrt{\left| N_{\alpha} \right|}
<16\left| N_{\alpha} \right|$. Combining this with \eqref{eq:ass-7aa}, we find
that \eqref{eq:ass-7aa} and \eqref{eq:ass-3} both hold if
\begin{equation}
\label{eq:C-2}
y_{m_{1}} \geq y_{k_{1}}
\geq \max \left( 16\left| N_{\alpha} \right|/\sqrt{d}, 4b^{2}\left| N_{\alpha} \right|/d \right).
\end{equation}

Since $y_{m_{1}} \geq 4\sqrt{\left| N_{\alpha} \right|/d}$ (from \eqref{eq:C-2}),
we obtain
\begin{equation}
\label{eq:25}
x_{m_{1}}^{2} = dy_{m_{1}}^{2}+N_{\alpha}
=dy_{m_{1}}^{2} \left( 1 + \frac{N_{\alpha}}{dy_{m_{1}}^{2}} \right)
\geq 0.9375dy_{m_{1}}^{2}.
\end{equation}

So
\begin{equation}
\label{eq:26}
\sqrt{x_{m_{1}}^{2}-N_{\alpha}}= \sqrt{dy_{m_{1}}^{2}}
< 1.04 \left| x_{m_{1}} \right|.
\end{equation}

Recall equation~(\eqELB{}) from \cite{V5} with $k$ there being $m_{1}$ here:
\begin{equation}
\label{eq:E-LB1}
E > \frac{0.1832|g|\cN_{d',4}\sqrt{d} \, y_{m_{1}}}{\left| N_{\alpha} \right|}.
\end{equation}

From $|g|\cN_{d',4} \geq 2$ (from Lemma~\lemGNda{} in
\cite{V7}) and $y_{m_{1}} \geq 16\left| N_{\alpha} \right|/\sqrt{d}$ (from \eqref{eq:C-2}),
we have $E>1$, as required for its use with Lemma~\ref{lem:2.1}.

Similarly, using equation~(\eqQLB{}) from \cite{V5}, we have
\begin{equation}
\label{eq:Q-LB1}
Q > \frac{2e^{1.68}\left( 1+\sqrt{0.9375} \right) \sqrt{d}y_{m_{1}}}{|g|\cN_{d',4}}
>\frac{21.12\sqrt{d} \, y_{m_{1}}}{|g|\cN_{d',4}}
\geq 21.12,
\end{equation}
where we use $y_{m_{1}} \geq |g|\cN_{d',4}/\sqrt{d}$ (from \eqref{eq:ass-3})
to establish the last inequality.

From equation~(\eqQUB{}) in \cite{V5}, we have
\begin{equation}
\label{eq:Q-UB2}
Q <\frac{21.47\sqrt{d} \, y_{m_{1}}}{|g|\cN_{d',4}}.
\end{equation}

Writing $\omega_{m_{1}}=e^{i\varphi_{m_{1}}}$, with $-\pi < \varphi_{m_{1}} \leq \pi$,
from the expression for $\ell_{0}$ in \eqref{eq:ell0-defn} and from Lemma~\ref{lem:zeta}(a),
we can take
\begin{equation}
\label{eq:ell-UB}
\ell_{0}=0.2685 \left| \varphi_{m_{1}} \right|<0.615\sqrt{\left| N_{\alpha} \right|}/\left| x_{m_{1}} \right|.
\end{equation}

Also from Lemma~\ref{lem:zeta}(a), we have $\left| \varphi_{m_{1}} \right|<0.6$,
so the condition $\left| \omega_{m_{1}}-1 \right|<1$ required in Lemma~\lemHypg{}\,
of \cite{V5} to apply the hypergeometric method is satisfied.

Let $q=x+y\sqrt{\core \left( N_{\alpha} \right)}=\left( r_{m_{1}}-s_{m_{1}}\sqrt{\core \left( N_{\alpha} \right)} \right)
\left( r_{m_{3}}+s_{m_{3}}\sqrt{\core \left( N_{\alpha} \right)} \right)$
and $p=x-y\sqrt{\core \left( N_{\alpha} \right)}$. Recall from (\eqQAbs{}) in
\cite{V5} that
\begin{equation}
\label{eq:qAbs}
|q| = \frac{\sqrt{f_{m_{1}}f_{m_{3}}} \left( y_{m_{1}}y_{m_{3}} \right)^{1/4}}{b}.
\end{equation}

Since $\core \left( \left| N_{\alpha} \right| \right)=2^{\ell}p^{m}$, where
$\ell,m \in \left\{ 0,1 \right\}$, from Lemma~\ref{lem:quad-rep} we have
\begin{equation}
\label{eq:f1f3UB}
f_{m_{1}}f_{m_{3}} \leq 16b^{4}.
\end{equation}

We are now ready to deduce the required contradiction from the assumptions above.
With $r_{0}$ as in Lemma~\ref{lem:2.1} using $\theta=\omega_{m_{1}}^{1/4}$ and
$c=3/4$, we follow the same process as
in \cite{V7}. In each of the four main steps, we will obtain an upper bound for
$y_{m_{1}}$, which must hold in order for the conditions defining that step to
hold. Thus we get the desired lower bound for $y_{m_{1}}$ in the statement of
Theorem~\ref{thm:1.2-seq-new} (since we want none of these four steps to be possible).

In Section~\ref{sect:proof-thm13}, we combine the lower bounds for $y_{m_{1}}$ in
each step with \eqref{eq:ym1-LB}. In this way, we get a contradiction in each step
unless $d$ is bounded from above, with an explicit bound for $d$.

Although it is not explicit in our notation, we note here that $g$, $d'$, $\cN_{d',4}$,
$p_{r}$ and $q_{r}$ all depend on the index $m_{1}$ that we have specified above.

\subsection{$r_{0}=1$ and $\eta pq_{r_{0}} \neq p_{r_{0}}q$}
\label{subsect:thm12-step-i}

We start with the upper bound for $y_{m_{3}}$ in equation~(\eqYUBStepOne{}) in
\cite{V7}:
\[
y_{m_{3}}^{3} < 2180(b(1-c))^{-4} \left( N_{\alpha}f_{m_{1}}f_{m_{3}} \right)^{2}y_{m_{1}}^{5}.
\]

Applying \eqref{eq:f1f3UB}, we obtain
\begin{equation}
\label{eq:y3UB}
y_{m_{3}}^{3} < 559000(1-c)^{-4} N_{\alpha}^{2}b^{4}y_{m_{1}}^{5}.
\end{equation}

We will now combine \eqref{eq:y3UB} with the gap principle in Lemma~\ref{lem:gap}
to show that this case cannot occur. Applying Lemma~\ref{lem:gap} twice, we
have
\[
y_{m_{3}}> \frac{15d^{2}}{b^{4} \left| N_{\alpha} \right|^{2}} y_{m_{2}}^{3}
> \left( \frac{15d^{2}}{b^{4} \left| N_{\alpha} \right|^{2}} \right)^{4} y_{m_{1}}^{9}.
\]

Combining the lower bound that this provides for $y_{m_{3}}^{3}$ with the upper
bound for $y_{m_{3}}^{3}$ in \eqref{eq:y3UB} and cancelling the common factor
of $y_{m_{1}}^{5}$ on both sides, we find that
\[
\left( \frac{15d^{2}}{b^{4} \left| N_{\alpha} \right|^{2}} \right)^{12}y_{m_{1}}^{22}
< 559000(1-c)^{-4}N_{\alpha}^{2}b^{4},
\]
provided that \eqref{eq:ass-7aa} holds. If \eqref{eq:C-2} holds, then so does
\eqref{eq:ass-7aa}.

We need \eqref{eq:ass-5aa-xxx} here, along with \eqref{eq:ass-7aa}, as they,
together with our choice of $m_{1}$, $m_{2}$ and $m_{3}$, imply that $m_{1}, m_{3} \neq 0$,
as required by Lemma~\ref{lem:gap} (which we use above).

Rearranging, taking $22$nd roots of both sides and using $c=0.75$ (a choice we
justified in Subsection~\subSectStepiii{} of \cite{V5}), we find that
\begin{equation}
\label{eq:yk-UB-step1}
\frac{0.54 b^{26/11}\left| N_{\alpha} \right|^{13/11}}{d^{12/11}}>y_{m_{1}}
\geq y_{k_{1}}
\end{equation}
must hold in this case.

\subsection{$r_{0}=1$ and $\eta pq_{r_{0}}= p_{r_{0}}q$}
\label{subsect:thm12-step-ii}

As in Subsection~\ref{subsect:thm12-step-i}, we start by proving an upper bound for
$y_{m_{3}}$ that holds for all $r_{0} \geq 1$ with $\zeta_{4}p/q=p_{r_{0}}/q_{r_{0}}$
for some $4$th root of unity, $\zeta_{4}$.

From (\eqYLUBUsingYK{}) in \cite{V5}
with $k$ and $\ell$ there set to $m_{1}$ and $m_{3}$, respectively, we have
\begin{equation}
\label{eq:yEll-UB-step2}
1.734r_{0}^{1/2} \left(4\frac{d}{\left| N_{\alpha} \right|} \right)^{r_{0}}y_{m_{1}}^{2r_{0}+1}
>y_{m_{3}}.
\end{equation}

Specialising this inequality to the case of $r_{0}=1$ and assuming \eqref{eq:ass-5aa-xxx}
and \eqref{eq:C-2}, we can apply the gap principle in Lemma~\ref{lem:gap}
twice to show that
\[
6.94 \frac{d}{\left| N_{\alpha} \right|} y_{m_{1}}^{3}>y_{m_{3}}
>15 \frac{d^{2}}{b^{4} \left| N_{\alpha} \right|^{2}} y_{m_{2}}^{3}
> \left( 15 \frac{d^{2}}{b^{4} \left| N_{\alpha} \right|^{2}} \right)^{4} y_{m_{1}}^{9}.
\]

Rearranging this, we find that
\[
\frac{0.00014b^{16}\left| N_{\alpha} \right|^{7}}{d^{7}}>y_{m_{1}}^{6}.
\]

Taking $6$th roots,
\begin{equation}
\label{eq:yk-UB-step2}
\frac{0.23b^{8/3}\left| N_{\alpha} \right|^{7/6}}{d^{7/6}}>y_{m_{1}}
\geq y_{k_{1}}
\end{equation}
must hold.

\subsection{$r_{0}>1$ and $\eta pq_{r_{0}} \neq p_{r_{0}}q$}
\label{subsect:thm12-step-iii}

Recall equation~(\eqYUBStepThree{}) from \cite{V7}:
\begin{equation}
\label{eq:y1UB-step3-b}
\left( f_{m_{1}}f_{m_{3}} \right)^{8}
\frac{0.181\left| N_{\alpha} \right|^{2} \left( |g|\cN_{d',4} \right)^{4}}{b^{16}d^{4}}
> \left( \frac{y_{m_{1}}^{4}d^{2}\left( |g|\cN_{d',4} \right)^{8}}
{59.2^{4} \left| N_{\alpha} \right|^{6}} \right)^{2r_{0}-1}.
\end{equation}

Applying \eqref{eq:f1f3UB}, the left-hand side of \eqref{eq:y1UB-step3-b}
is at most
\begin{equation}
\label{eq:ykUB-step3a}
b^{16}\frac{777.4 \cdot 10^{6} \left| N_{\alpha} \right|^{2} \left( |g|\cN_{d',4} \right)^{4}}{d^{4}}.
\end{equation}

So we need
\begin{equation}
\label{eq:yk-UB-step3}
\frac{59.2 \cdot 167^{1/(2r_{0}-1)}\left| N_{\alpha} \right|^{3/2+1/(2(2r_{0}-1))} b^{4/(2r_{0}-1)}}
{d^{1/2+1/(2r_{0}-1)} \left( |g|\cN_{d',4} \right)^{2-1/(2r_{0}-1)}}
>y_{m_{1}}
\end{equation}
to hold in order for \eqref{eq:ykUB-step3a} to exceed the right-hand side of
\eqref{eq:y1UB-step3-b}.

We now simplify this bound. First, by Lemma~\lemGNda{} in \cite{V7}, we know that
$|g|\cN_{d',4} \geq 2$. From this and $r_{0} \geq 2$, we have
$\left( |g|\cN_{d',4} \right)^{-2+1/(2r_{0}-1)} \leq 2^{-2+1/(2r_{0}-1)}$. Applying
this, followed by $d \geq 2$ and $r_{0} \geq 2$, we obtain
\begin{align*}
& \frac{59.2 \cdot 167^{1/(2r_{0}-1)}\left| N_{\alpha} \right|^{3/2+1/(2(2r_{0}-1))} b^{4/(2r_{0}-1)}}
{d^{1/2+1/(2r_{0}-1)} \left( |g|\cN_{d',4} \right)^{2-1/(2r_{0}-1)}}
\leq
\frac{59.2 b^{4/(2r_{0}-1)} \left| N_{\alpha} \right|^{3/2+1/(2(2r_{0}-1))}}
{4d^{1/2}} \cdot \left( \frac{167 \cdot 2}{d} \right)^{1/(2r_{0}-1)} \\
<& \frac{82b^{4/3} \left| N_{\alpha} \right|^{5/3}}{d^{1/2}}.
\end{align*}

So
\begin{equation}
\label{eq:yk-UB-step-iii}
\frac{82b^{4/3} \left| N_{\alpha} \right|^{5/3}}{d^{1/2}}>y_{m_{1}}
\geq y_{k_{1}}
\end{equation}
must hold.

\subsection{$r_{0}>1$ and $\eta pq_{r_{0}}= p_{r_{0}}q$}
\label{subsect:thm12-step-iv}

Recall equation~(\eqUBStepFour{}) from \cite{V7}:
\begin{equation}
\label{eq:yk-UB-step4-init}
1
> \frac{0.1911b^{4}d}{f_{m_{1}}^{2}f_{m_{3}}^{2} \left| N_{\alpha} \right|}
  \left( \frac{0.0002344 |g|^{4}\cN_{d',4}^{4}}{\left| N_{\alpha} \right|^{3}} dy_{m_{1}}^{2} \right)^{r_{0}-1}.
\end{equation}

We now proceed similarly to the way we did in Subsection~\ref{subsect:thm12-step-iii}.

Applying \eqref{eq:f1f3UB}, we obtain
\[
\frac{1340b^{4}\left| N_{\alpha} \right|}{d}
> \left( \frac{0.0002344 |g|^{4}\cN_{d',4}^{4}}{\left| N_{\alpha} \right|^{3}} dy_{m_{1}}^{2} \right)^{r_{0}-1}.
\]

In order for this inequality to hold, we must have
\begin{equation}
\label{eq:yk-UB-step4}
65.32\frac{b^{2/(r_{0}-1)}\left| N_{\alpha} \right|^{3/2+1/(2(r_{0}-1))} \cdot 1340^{1/(2(r_{0}-1))}}
{|g|^{2}\cN_{d',4}^{2} d^{1/2+1/(2(r_{0}-1))}}>y_{m_{1}}.
\end{equation}

Applying $|g|\cN_{d',4} \geq 2$ (from Lemma~\lemGNda{} in \cite{V7}), followed by
$d \geq 2$ and $r_{0} \geq 2$, we obtain
\begin{align*}
& 65.32\frac{b^{2/(r_{0}-1)}\left| N_{\alpha} \right|^{3/2+1/(2(r_{0}-1))} \cdot 1340^{1/(2(r_{0}-1))}}
{|g|^{2}\cN_{d',4}^{2} d^{1/2+1/(2(r_{0}-1))}}
\leq 16.33\frac{b^{2/(r_{0}-1)}\left| N_{\alpha} \right|^{3/2+1/(2(r_{0}-1))}}
{d^{1/2}} \left( \frac{1340}{d} \right)^{1/(2(r_{0}-1))}\\
<& \frac{423 b^{2}\left| N_{\alpha} \right|^{2}}{d^{1/2}}.
\end{align*}

Hence the above upper bound for $y_{m_{1}}$ holds if
\begin{equation}
\label{eq:yk-UB-step-iv}
\frac{423b^{2}\left| N_{\alpha} \right|^{2}}{d^{1/2}}>y_{m_{1}}
\geq y_{k_{1}}.
\end{equation}

\subsection{Consolidation}
\label{subsect:thm12-consolidation}

Recall that we have assumed that there are five distinct squares,
$y_{k_{5}}>y_{k_{4}}>y_{k_{3}}>y_{k_{2}}>y_{k_{1}}$ with $k_{1} \geq 1$ or $k_{1} \leq K$.
Also, by \eqref{eq:C-2}, we have
\[
y_{k_{1}} \geq \max \left( 16\left| N_{\alpha} \right|/\sqrt{d}, 4b^{2}\left| N_{\alpha} \right|/d \right).
\]

But from \eqref{eq:yk-UB-step1}, \eqref{eq:yk-UB-step2}, \eqref{eq:yk-UB-step-iii}
and \eqref{eq:yk-UB-step-iv}, we see that if we also have
\[
y_{k_{1}}> \max \left(
\begin{array}{l}
\displaystyle\frac{0.54 b^{26/11}\left| N_{\alpha} \right|^{13/11}}{d^{12/11}},
\displaystyle\frac{0.23b^{8/3}\left| N_{\alpha} \right|^{7/6}}{d^{7/6}},
\displaystyle\frac{82b^{4/3} \left| N_{\alpha} \right|^{5/3}}{d^{1/2}},
\displaystyle\frac{423b^{2}\left| N_{\alpha} \right|^{2}}{d^{1/2}}
\end{array}
\right),
\]
then we cannot be in any of the four cases in Subsections~\ref{subsect:thm12-step-i}--\ref{subsect:thm12-step-iv}.
Hence we get a contradiction to our assumption that there are five such distinct
squares if
\[
y_{k_{1}}> \max \left(
\frac{0.54 b^{26/11}\left| N_{\alpha} \right|^{13/11}}{d^{12/11}},
\frac{0.23b^{8/3}\left| N_{\alpha} \right|^{7/6}}{d^{7/6}},
\frac{82b^{4/3} \left| N_{\alpha} \right|^{5/3}}{d^{1/2}},
\displaystyle\frac{423b^{2}\left| N_{\alpha} \right|^{2}}{d^{1/2}},
\frac{16\left| N_{\alpha} \right|}{\sqrt{d}},
\frac{4b^{2}\left| N_{\alpha} \right|}{d}
\right).
\]

Since $b, \left| N_{\alpha} \right|, d \geq 1$, taking the largest constant, the
largest exponent on $\left| N_{\alpha} \right|$, the largest exponent on $b$ and
the smallest exponent on $d$ from the terms in this max, we get a contradiction
if
\[
y_{k_{1}}> \frac{423b^{8/3}\left| N_{\alpha} \right|^{2}}{\sqrt{d}},
\]
proving Theorem~\ref{thm:1.2-seq-new}.

In fact, we have proven the following more precise version of Theorem~\ref{thm:1.2-seq-new}.
We will use it below in Subsection~\ref{subsect:proof-thm13a} for the proof of
Theorem~\ref{thm:1.3-seq-new}(a). For the case where \(-N_\alpha\) is a square,
we can use Proposition~4.1 of \cite{V7}.

\begin{proposition}
\label{prop:1.2-seq-new}
Let $a$, $b$ and $d$ be positive integers, where $d$ is not a square, $N_{\alpha}<0$
and $N_{\alpha}=-2^{\ell}p^{m}$ with $p$ an odd prime and $\ell$, $m$ non-negative
integers. There are at most four distinct squares among the
$y_{k}$'s with $k \geq 1$ or $k \leq K$, and
\[
y_{k} > \max \left(
\frac{0.54 b^{26/11}\left| N_{\alpha} \right|^{13/11}}{d^{12/11}},
\frac{0.23b^{8/3}\left| N_{\alpha} \right|^{7/6}}{d^{7/6}},
\frac{82b^{4/3} \left| N_{\alpha} \right|^{5/3}}{d^{1/2}},
\displaystyle\frac{423b^{2}\left| N_{\alpha} \right|^{2}}{d^{1/2}},
\frac{16\left| N_{\alpha} \right|}{\sqrt{d}},
\frac{4b^{2}\left| N_{\alpha} \right|}{d}
\right).
\]
\end{proposition}

\section{Proof of Theorem~\ref{thm:1.3-seq-new}}
\label{sect:proof-thm13}

Here we will combine our bounds on $y_{k}$ in Section~\ref{sect:proof-thm12}
with \eqref{eq:ym1-LB} to obtain Theorem~\ref{thm:1.3-seq-new}. As in
Section~\ref{sect:proof-thm12}, we proceed in separate steps.
Also as in that section (see \eqref{eq:C-2}), here we will assume that
there are five distinct squares,
\[
y_{k_{5}}>y_{k_{4}}>y_{k_{3}}>y_{k_{2}}>y_{k_{1}}
\geq \max \left( 16\left| N_{\alpha} \right|/\sqrt{d}, 4b^{2}\left| N_{\alpha} \right|/d \right)
\]
among the $y_{k}$'s. However here we will assume that
\begin{equation}
\label{eq:ass-5aa-yyy}
k_{1} \geq 2 \text{ or } k_{1} \leq K-1.
\end{equation}
In this way, we will obtain a contradiction once $d$ is sufficiently large.

The reason for the stronger condition on $k_{1}$ in \eqref{eq:ass-5aa-yyy},
rather than the condition in \eqref{eq:ass-5aa-xxx}, is to obtain a stronger
lower bound for $y_{k_{1}}$
from Lemma~\ref{lem:Y-LB1}(b). This will provide us with better lower bounds
for $d$.

In order to apply Lemma~\ref{lem:Y-LB1}(b), we will also assume that
\begin{equation}
\label{eq:ass-dLB-from-yLB}
d \geq \frac{300}{u^{2}}.
\end{equation}

Applying \eqref{eq:ass-5aa-yyy} and \eqref{eq:ass-dLB-from-yLB} to Lemma~\ref{lem:Y-LB1}(b),
we obtain
\begin{equation}
\label{eq:ym1-LB}
y_{k_{1}} \geq \frac{0.99|N_{\alpha}|du^{4}}{4b^{2}}.
\end{equation}

\subsection{$r_{0}=1$ and $\eta pq_{r_{0}} \neq p_{r_{0}}q$}
\label{subsect:thm13-step-i}

In Subsection~\ref{subsect:thm12-step-i}, we showed that
\[
y_{k_{1}} < \frac{0.54 b^{26/11}\left| N_{\alpha} \right|^{13/11}}{d^{12/11}}.
\]

This is \eqref{eq:yk-UB-step1}. It contradicts \eqref{eq:ym1-LB} if
\begin{equation}
\label{eq:ass-6a}
d>\frac{1.46 \left| N_{\alpha} \right|^{2/23} b^{48/23}}{u^{44/23}}.
\end{equation}

It is useful to get a lower bound for $d$ just in terms of $b$ and $u$.
Applying $\left| N_{\alpha} \right| < db^{4}$ to the right-hand side of \eqref{eq:ass-6a}
and solving the resulting inequality for $d$, we get a contradiction if
\begin{equation}
\label{eq:ass-6b}
d>\frac{1.52 b^{8/3}}{u^{44/21}}.
\end{equation}

Hence we cannot be in this step if either \eqref{eq:ass-6a} or \eqref{eq:ass-6b}
holds, along with the earlier assumptions made at the start of this section.

\subsection{$r_{0}=1$ and $\eta pq_{r_{0}}=p_{r_{0}}q$}
\label{subsect:thm13-step-ii}

Combining \eqref{eq:yk-UB-step2} with \eqref{eq:ym1-LB}, we get a contradiction
provided that
\begin{equation}
\label{eq:ass-9a}
d > \frac{0.97\left| N_{\alpha} \right|^{1/13}b^{28/13}}{u^{24/13}}.
\end{equation}

Applying $\left| N_{\alpha} \right|<db^{4}$, we obtain
\begin{equation}
\label{eq:ass-9b}
d > \frac{0.97b^{8/3}}{u^{2}}.
\end{equation}
%

\subsection{$r_{0}>1$ and $\eta pq_{r_{0}} \neq p_{r_{0}}q$}
\label{subsect:thm13-step-iii}

Combining \eqref{eq:yk-UB-step3} and \eqref{eq:ym1-LB}, if
\[
\frac{0.99|N_{\alpha}|du^{4}}{4b^{2}}
>\frac{59.2 \cdot 167^{1/(2r_{0}-1)}\left| N_{\alpha} \right|^{3/2+1/(2(2r_{0}-1))} b^{4/(2r_{0}-1)}}
{d^{1/2+1/(2r_{0}-1)} \left( |g|\cN_{d',4} \right)^{2-1/(2r_{0}-1)}},
\]
then the right-hand side of \eqref{eq:y1UB-step3-b} is greater than \eqref{eq:ykUB-step3a}.

Since $|g|\cN_{d',4} \geq 2$ (from Lemma~\lemGNda{} in \cite{V7}), this inequality
holds if
\begin{equation}
\label{eq:dLB-step3a}
d >\frac{240^{2(2r_{0}-1)/(6r_{0}-1)} \cdot 167^{2/(6r_{0}-1)} \left| N_{\alpha} \right|^{2r_{0}/(6r_{0}-1)} b^{4(2r_{0}+1)/(6r_{0}-1)}}
{u^{8(2r_{0}-1)/(6r_{0}-1)} 2^{2(4r_{0}-3)/(6r_{0}-1)}}.
\end{equation}

For $r_{0} \geq 2$, we have the following:\\
$2r_{0}/ \left( 6r_{0}-1 \right)$ decreases from $4/11=0.3636\ldots$ towards $1/3$,
so $\left| N_{\alpha} \right|^{2r_{0}/(6r_{0}-1)} \leq \left| N_{\alpha} \right|^{4/11}$,\\
$4 \left( 2r_{0}+1 \right)/ \left( 6r_{0}-1 \right)$ decreases from $20/11$
towards $4/3$,
so $b^{4(2r_{0}+1)/(6r_{0}-1)} \leq b^{20/11}$,\\
$8 \left( 2r_{0}-1 \right)/ \left( 6r_{0}-1 \right)$ increases from $24/11$
towards $8/3$,
so $u^{8(2r_{0}-1)/(6r_{0}-1)} \geq u^{24/11}$.

Also, $2 \left( 2r_{0}-1 \right)/ \left( 6r_{0}-1 \right)=2/3-4/\left( 3 \left( 6r_{0}-1 \right) \right)$
and $2 \left( 4r_{0}-3 \right)/ \left( 6r_{0}-1 \right)=4/3-14/\left( 3 \left( 6r_{0}-1 \right) \right)$,
so
\[
\frac{240^{2(2r_{0}-1)/(6r_{0}-1)} \cdot 167^{2/(6r_{0}-1)}}{2^{2(4r_{0}-3)/(6r_{0}-1)}}
=\frac{240^{2/3}}{2^{4/3}}
\left( \frac{167^{6} \cdot 2^{14}}{240^{4}} \right)^{1/(3(6r_{0}-1))}.
\]

Since $167^{6} \cdot 2^{14}/240^{4}>1$, we find that
$\left( 167^{6} \cdot 2^{14}/240^{4} \right)^{1/(3(6r_{0}-1))}$
takes its maximum value at $r_{0}=2$, where it is less than $1.752$.
Since $240^{2/3}/2^{4/3}<15.4$, \eqref{eq:dLB-step3a} holds if
\begin{equation}
\label{eq:ass-10a}
d >\frac{27 b^{20/11} \left| N_{\alpha} \right|^{4/11}}{u^{24/11}}.
\end{equation}

If we apply $\left| N_{\alpha} \right|<db^{4}$ to \eqref{eq:dLB-step3a}, we obtain
\[
d^{(4r_{0}-1)/(6r_{0}-1)} >\frac{240^{2(2r_{0}-1)/(6r_{0}-1)} \cdot 167^{2/(6r_{0}-1)} b^{(16r_{0}+4)/(6r_{0}-1)}}
{u^{8(2r_{0}-1)/(6r_{0}-1)} 2^{2(4r_{0}-3)/(6r_{0}-1)}}.
\]

That is,
\[
d>\frac{240^{2(2r_{0}-1)/(4r_{0}-1)} \cdot 167^{2/(4r_{0}-1)} b^{(16r_{0}+4)/(4r_{0}-1)}}
{u^{8(2r_{0}-1)/(4r_{0}-1)} 2^{2(4r_{0}-3)/(4r_{0}-1)}}.
\]

We have
\begin{align*}
& 240^{2(2r_{0}-1)/(4r_{0}-1)} \cdot 167^{2/(4r_{0}-1)} / 2^{2(4r_{0}-3)/(4r_{0}-1)}
=240^{1-1/(4r_{0}-1)} \cdot 167^{2/(4r_{0}-1)} / 2^{2-4/(4r_{0}-1)} \\
= & (240/4) \left( 167^{2} \cdot 16/240 \right)^{1/(4r_{0}-1)}
<176,
\end{align*}
with the maximum value for $r_{0} \geq 2$ occurring at $r_{0}=2$.

So
\[
d > \frac{176b^{4+8/(4r_{0}-1)}}
{u^{4-4/(4r_{0}-1)}}.
\]

Since $r_{0} \geq 2$, the desired lower bound for $y_{k_{1}}$ holds if
\begin{equation}
\label{eq:ass-10b}
d>\frac{176 b^{36/7}}{u^{24/7}}.
\end{equation}

\subsection{$r_{0}>1$ and $\eta pq_{r_{0}}= p_{r_{0}}q$}
\label{subsect:thm13-step-iv}


As in the previous subsections, using \eqref{eq:ym1-LB}, we want to show that
$0.99\left| N_{\alpha} \right| du^{4}/ \left( 4b^{2} \right)$
is larger than the upper bound for $y_{k_{1}}$ in \eqref{eq:yk-UB-step4}.
That is, we want to show that
\[
d^{3/2+1/(2(r_{0}-1))}
>264\frac{b^{2+2/(r_{0}-1)}\left| N_{\alpha} \right|^{1/2+1/(2(r_{0}-1))} \cdot 1340^{1/(2(r_{0}-1))}}
{u^{4}|g|^{2}\cN_{d',4}^{2}}.
\]

This becomes
\begin{equation}
\label{eq:dLB-step4a}
d>264^{2/3-2/(3(3r_{0}-2))}\frac{b^{4/3+8/(3(3r_{0}-2))}\left| N_{\alpha} \right|^{1/3+2/(3(3r_{0}-2))} \cdot 1340^{1/(3r_{0}-2)}}
{u^{8/3-8/(3(3r_{0}-2))}\left( |g|\cN_{d',4} \right)^{4/3-4/(3(3r_{0}-2))}}.
\end{equation}

As in Subsection~\ref{subsect:thm12-step-iii}, for $r_{0} \geq 2$,\\
$4/3+8/(3(3r_{0}-2)) \leq 2$, with the maximum value occurring for $r_{0}=2$,
so $b^{4/3+8/(3(3r_{0}-2))} \leq b^{2}$,\\
$1/3+2/(3(3r_{0}-2)) \leq 1/2$, with the maximum value occurring for $r_{0}=2$,
so $\left| N_{\alpha} \right|^{1/3+2/(3(3r_{0}-2))} \leq \left| N_{\alpha} \right|^{1/2}$,\\
$8/3-8/(3(3r_{0}-2)) \geq 2$, with the minimum value occurring for $r_{0}=2$,
so $u^{8/3-8/(3(3r_{0}-2))} \geq u^{2}$.

Combining this with $|g|\cN_{d',4} \geq 2$ (from Lemma~\lemGNda{} in \cite{V7})
and $r_{0} \geq 2$, our lower bound for $d$ holds if
\[
d>\frac{264^{2/3}}{2^{4/3}} 
\left( \frac{2^{4/3} \cdot 1340}{264^{2/3}} \right)^{1/(3r_{0}-2)} \frac{b^{2} \left| N_{\alpha} \right|^{1/2}}{u^{2}}.
\]

Since $264^{2/3}/2^{4/3}<16.34$ and for $r_{0} \geq 2$, we have
$\left( 2^{4/3} \cdot 1340/264^{2/3} \right)^{1/(3r_{0}-2)}<81.9^{1/(3r_{0}-2)}
\leq 82.1^{1/4}<3.02$, the maximum value of the constant term is less than $50$.
Hence
\begin{equation}
\label{eq:ass-11a}
d>\frac{50b^{2}\left| N_{\alpha} \right|^{1/2}}{u^{2}}.
\end{equation}

If we apply $\left| N_{\alpha} \right|<db^{4}$ to \eqref{eq:dLB-step4a}, we obtain
\[
d>264^{2/3-2/(3(3r_{0}-2))}\frac{b^{4/3+8/(3(3r_{0}-2))}\left( db^{4} \right)^{1/3+2/(3(3r_{0}-2))} \cdot 1340^{1/(3r_{0}-2)}}
{u^{8/3-8/(3(3r_{0}-2))}\left( |g|\cN_{d',4} \right)^{4/3-4/(3(3r_{0}-2))}}.
\]

That is,
\[
d^{2(r_{0}-1)/(3r_{0}-2)}
> 264^{2/3-2/(3(3r_{0}-2))}\frac{b^{8r_{0}/(3r_{0}-2)} \cdot 1340^{1/(3r_{0}-2)}}
{u^{8/3-8/(3(3r_{0}-2))}\left( |g|\cN_{d',4} \right)^{4/3-4/(3(3r_{0}-2))}}.
\]

So,
\[
d>\frac{264 \cdot 1340^{1/(2(r_{0}-1))} b^{4+4/(r_{0}-1)}}
{u^{4}\left( |g|\cN_{d',4} \right)^{2}}.
\]

Since $r_{0} \geq 2$ and $|g| \cN_{d',4} \geq 2$ (from Lemma~\lemGNda{} in \cite{V7}),
the desired lower bound for $y_{k_{1}}$ holds if
\begin{equation}
\label{eq:ass-11b}
d > \frac{2416b^{8}}{u^{4}}.
\end{equation}


\subsection{Proof of Theorem~\ref{thm:1.3-seq-new}(b)}
\label{subsect:step-v}

We now bring together all of the conditions that we have imposed to obtain a
contradiction from the assumption that for $d$ sufficiently large, there are
five distinct squares with $k \geq 2$ or $k \leq K-1$.

First, to be able to apply \eqref{eq:ym1-LB}, we need
\begin{equation}
\label{eq:dUB-from-ass-1a}
d \geq \frac{300}{u^{2}}.
\end{equation}

For \eqref{eq:C-2}, we need $y_{k_{1}} \geq 4b^{2}\left| N_{\alpha} \right|/d$.
Using \eqref{eq:ym1-LB}, we need to show that
$0.99\left| N_{\alpha} \right| du^{4}/ \left( 4b^{2} \right)>4b^{2}\left| N_{\alpha} \right|/d$.
This holds if
\begin{equation}
\label{eq:dUB-from-ass-2b}
d \geq 4.1 \frac{b^{2}}{u^{2}}.
\end{equation}

Also for \eqref{eq:C-2}, we need 
$y_{k_{1}} \geq 16\left| N_{\alpha} \right|/\sqrt{d}$.
By \eqref{eq:ym1-LB}, this holds if
\begin{equation}
\label{eq:dUB-from-ass-2a}
d \geq \frac{16.2b^{4/3}}{u^{8/3}}.
\end{equation}

In Subsection~\ref{subsect:thm13-step-i}, we added the assumption that
$d>1.46 b^{48/23}\left| N_{\alpha} \right|^{2/23}/u^{44/23}$ in \eqref{eq:ass-6a}.

In Subsection~\ref{subsect:thm13-step-ii}, we added the assumption that
$d>0.97b^{28/13}\left| N_{\alpha} \right|^{1/13}/u^{24/13}$ in \eqref{eq:ass-9a}.

At the end of Subsection~\ref{subsect:thm13-step-iii}, we imposed the condition that
$d>27b^{20/11} \left| N_{\alpha} \right|^{4/11}/u^{24/11}$ in \eqref{eq:ass-10a}.

Finally, at the end of Subsection~\ref{subsect:thm13-step-iv}, we imposed the condition
that $d>50b^{2}\left| N_{\alpha} \right|^{1/2}/u^{2}$ in \eqref{eq:ass-11a}.

Rounding up some of the constants in these lower bounds for $d$ and then combining them,
we have
\begin{equation}
\label{eq:d-LB2}
d \geq \max \left(
\frac{300}{u^{2}},
\frac{5b^{2}}{u^{2}},
\frac{17b^{4/3}}{u^{8/3}},
\frac{1.46b^{48/23} \left| N_{\alpha} \right|^{2/23}}{u^{44/23}}, 
\frac{b^{28/13}\left| N_{\alpha} \right|^{1/13}}{u^{24/13}},
\frac{27 b^{20/11} \left| N_{\alpha} \right|^{4/11}}{u^{24/11}},
\frac{50b^{2}\left| N_{\alpha} \right|^{1/2}}{u^{2}}.
\right)
\end{equation}

Since we are assuming that $b \geq 8$, we have $300/u^{2}<5b^{2}/u^{2}$.
Since $8^{28/13-2}=1.377\ldots$, $50b^{2}\left| N_{\alpha} \right|^{1/2}/u^{2}
<37b^{28/13}\left| N_{\alpha} \right|^{1/2}/u^{2}$. With these two replacements,
the largest coefficient in the resulting maximum is $37$.
The largest exponent on $b$ in the terms in this max is $28/13$. 
Similarly, taking the largest exponent on $\left| N_{\alpha} \right|$
as well as the smallest exponent on $u$ in their denominators provides the lower
bound for $d$ in Theorem~\ref{thm:1.3-seq-new}(b), completing its proof.

\subsection{Proof of Theorem~\ref{thm:1.3-seq-new}(a)}
\label{subsect:proof-thm13a}

Here we proceed as in Subsection~\subsectProofThmForD{} of \cite{V7}.

As in Subsection~\ref{subsect:step-v}, we are going to obtain a contradiction
from the assumption that there are
five distinct squares $y_{k}$ with $k \geq 2$ or $k \leq K-1$.

To complete the proof of Theorem~\ref{thm:1.3-seq-new}(a), we need to remove the
assumption on $d$ in \eqref{eq:d-LB2} in Subsection~\ref{subsect:step-v} for
$b=1,\ldots,7$.

If instead of \eqref{eq:ass-6a}, \eqref{eq:ass-9a}, \eqref{eq:ass-10a} and
\eqref{eq:ass-11a} that we used to obtain \eqref{eq:d-LB2}, we use \eqref{eq:ass-6b},
\eqref{eq:ass-9b}, \eqref{eq:ass-10b} and \eqref{eq:ass-11b} (rounding up the
first two constants slightly), then we obtain
\begin{equation}
\label{eq:d-LB3}
d \geq \max \left(
\frac{300}{u^{2}},
\frac{5b^{2}}{u^{2}},
\frac{17b^{4/3}}{u^{8/3}},
\frac{1.52 b^{8/3}}{u^{44/21}},  
\frac{b^{8/3}}{u^{2}}, 
\frac{176b^{36/7}}{u^{24/7}},
\frac{2416b^{8}}{u^{4}}
\right).
\end{equation}

We use \eqref{eq:d-LB3} to get initial upper bounds, $D_{b,u}$, for the values of
$d$ that we need to check. For each value of $b$, the maximum value of $D_{b,u}$
comes from $u=1$. We put $D_{b} = D_{b,1}$.

For each value of $b$, we find the smallest value of $u$ such that the right-hand
side of \eqref{eq:d-LB3} is less than $2$ and denote this value of $u$ by $U_{b}$.

For each pair $(b,u)$ with $1 \leq u< U_{b}$, consider all integers
$d=\left( t^{2} \pm 4 \right)/ u^{2} \leq D_{b,u}$.

For each such $(b,d,t,u)$, we find all positive integers $a$ such that $-N_{\alpha}
=db^{4}-a^{2}$ can be written as $2^{\ell}p^{m}$, where $p$ is an odd prime and $\ell$,
$m$ are non-negative integers. In fact, checking this can take a considerable amount
of time compared to just checking the sequence associated with a tuple, $(a,b,d,t,u)$,
so instead we check a weaker condition. If the odd part of $N_{\alpha}$ is divisible
by at least one odd prime, $p \leq 61$, and has another factor greater than $1$
and relatively prime to $p$, then we eliminate it. Otherwise, we consider the
tuple, $(a,b,d,t,u)$.

We found $c_{1,b}$ such tuples $(a,b,d,t,u)$, where $c_{1,b}$ is as in Table~\ref{table:b-data}.
For none of them does the lower bound for $d$ in \eqref{eq:d-LB3} hold. Instead
we use Proposition~\ref{prop:1.2-seq-new}, which applies to all relevant sequences,
to treat them.

From Proposition~\ref{prop:1.2-seq-new}, we know that for each sequence, there
are at most four distinct squares with $y_{k}$ sufficiently large, where $k \geq 1$
or $k \leq K$. So for each of these tuples, $(a,b,d,t,u)$, for $k \geq 2$ or
$k \leq K-1$, we checked for squares among the $y_{k}$'s not satisfying the lower
bound on $y_{k}$ in Proposition~\ref{prop:1.2-seq-new}.
For each $b$, we found $c_{2,b}$ such tuples among the $c_{1,b}$ tuples considered
above.

As in \cite{V7}, for speed, we wrote the
code in the Java programming language, which we ran on a
Windows laptop with an Intel i7-13700H CPU and 32~GB of memory.
We record the CPU time for each value of $b$ in Table~\ref{table:b-data}.
The Java code can be found in the file
\verb!BAllNaPPSearch.java! in the \verb!java! subdirectory of the GitHub URL provided at the end of
Section~\ref{sect:intro}.

We checked each of the remaining $c_{2,b}$ sequences individually. First, we checked
that the odd part of $-N_{\alpha}$ was actually either $1$ or an odd prime power.
From \eqref{eq:pell}, squares in the associated sequences of $y_{k}$'s must arise
from positive integer solutions of $x^{2}-dy^{4}=N_{\alpha}$ (although note that
solutions of this equation may
arise from sequences other than the sequences of $y_{k}$'s that we are considering
here). We would like to find all integer solutions of these equations.
However, solving some of the resulting equations directly would take a very long
time because their coefficients are large.

To eliminate these large values of $d$ and $\left| N_{\alpha} \right|$, we first use
\eqref{eq:yk-UB-step3} with $|g|\cN_{d',4} \geq 2$ instead of \eqref{eq:yk-UB-step-iii}
and use \eqref{eq:yk-UB-step4} with $|g|\cN_{d',4} \geq 2$ instead of \eqref{eq:yk-UB-step-iv}.
Since these bounds decrease as $r_{0}$ increases and since for any fixed value of
$r_{0}$, these upper bounds for $y_{m_{1}}$ depend
only on the sequence and not on $y_{m_{1}}$ itself and since we know from Lemma~\ref{lem:Y-LB1}(a)
that the sequence is monotonic in both directions under the conditions here, this
allows us to obtain all possible values of $y_{m_{1}}$ for these remaining sequences.
For each of these values of $y_{m_{1}}$, we can use the proof of Proposition~\propRep{}
of \cite{V5} to determine the smallest value of $f_{m_{1}}$.
Using this value of $f_{m_{1}}$ in \eqref{eq:y1UB-step3-b} and \eqref{eq:yk-UB-step4-init}
leaves us with $206$ distinct tuples.
These checks were done using the PARI/GP \cite{Pari} code in
\verb!b1-7-nAPP-post-java-checks.gp! in the pari subdirectory of the GitHub URL.

We used the code in \verb!quartic_thue.gp! in \url{https://github.com/PV-314/hypgeom/general/pari}.
This code is described in Subsection~5.6 of \cite{V8}. Here it is executed for
the $206$ required curves in \verb!b1-7-nAPP-solve.gp! in the pari subdirectory
of the GitHub URL.
All but $6$ equations had at most $3$ solutions
in positive integers. The remaining
$6$ equations each had $4$ solutions in positive integers. Thus
Theorem~\ref{thm:1.3-seq-new}(a) holds.
In fact, one can also show that none of the $6$ remaining equations leads to
counterexamples to Conjecture~\ref{conj:1-seq}.
The $6$ equations have $\left( d, N_{\alpha} \right) \in
\left\{ (42,-38), (168, -152), (672, -608)\right.$,
$\left. (42,-608), (168,-2432), (672, -9728) \right\}$,
so they  actually all arise from $\left( d, N_{\alpha} \right)=(42,-38)$,
which has the solutions $(x,y)=(2, 1)$, $(58, 3)$, $(2858, 21)$, $(74198, 107)$.

\begin{center}
\begin{table}[h]
\begin{tabular}{|c|r|r|r|r|r|}\hline
\multicolumn{1}{|c|}{$b$} &
\multicolumn{1}{|c|}{$D_{b}$} &
\multicolumn{1}{|c|}{$U_{b}$} &
\multicolumn{1}{|c|}{$c_{1,b}$} &
\multicolumn{1}{|c|}{$c_{2,b}$} &
\multicolumn{1}{|c|}{\text{CPU time}} \\ \hline
%
%
%
%
%
%
%
  1  &           $2416$ &   $13$  &            $1421$ &   $10$    &       \text{$<$1s} \\ \hline
  2  &        $618,496$ &   $24$  &         $902,371$ &   $27$    &       \text{$<$1s} \\ \hline
  3  &     $15,851,376$ &   $54$  &      $44,478,366$ &   $71$    &       \text{$<$1s} \\ \hline
  4  &    $158,334,976$ &   $95$  &     $920,888,687$ &   $64$    &         \text{11s} \\ \hline
  5  &    $943,750,000$ &  $148$  &   $8,167,556,953$ &   $67$    &      \text{5m 51s} \\ \hline
  6  &  $4,057,952,256$ &  $213$  &  $45,539,991,918$ &  $148$    & \text{ 1h 29m 22s} \\ \hline
  7  & $13,927,759,216$ &  $289$  & $243,290,160,092$ &  $133$    & \text{13h  6m 41s} \\ \hline
\end{tabular}
\caption{Data for $b$}
\label{table:b-data}
\end{table}
\end{center}

\section{$x_{k}=0$}
\label{sect:xk-0}

As stated in Remark~\ref{rem:xk-0} in Subsection~\ref{subsect:notation}, we have
excluded until now the case where $x_{k}=0$ and $y_{k}$ is an integer square. Here we
treat this possibility. As noted in that remark, we can write $\alpha=b^{2} \sqrt{d}$.
So $y_{0}=b^{2}$ and $y_{1}=b^{2} \left( t^{2}+du^{2} \right)/4$.
Since $t^{2}+du^{2}=\pm 4+2du^{2}$, if $du$ is even, then $\left( t^{2}+du^{2} \right)/4$
is an integer and hence $b^{2}|y_{k}$. By our recurrence relation in \eqref{eq:yk-recurrence},
we find that $b^{2}$ divides all the $y_{k}$'s in this case.

If $du$ is odd, then $t$ must also be odd and $d \equiv 1 \pmod{4}$ (since
$t^{2}-du^{2}= \pm 4$). In this case, $\left( t^{2}+du^{2} \right)/4$
is not an integer, nor is $y_{2}/b^{2}$. However, one can show that $y_{3}/b^{2}$
is an integer and, by our recurrence relation, that $y_{k}/b^{2}$ is an integer
if and only if $3|k$. If $3 \nmid k$, then $2y_{k}/b^{2}$ is odd. Hence either
$y_{k} \notin \bbZ$ if $b$ is odd, or $v_{2} \left( y_{k} \right)
=2v_{2}(b)-1$ is odd if $b$ is even. Thus $y_{k}$ cannot be a square when $3 \nmid k$.

Whenever $y_{k}/b^{2} \in \bbZ$, put $y_{k}=b^{2}y_{k}'$. From \eqref{eq:pell}, we have
\[
x_{k}^{2}-db^{4}{ y_{k} '}^{2}=-db^{4}.
\]

Thus $b^{4}|x_{k}^{2}$ and writing $x_{k}=b^{2}x_{k}'$, we have ${x_{k}'}^{2}-d{ y_{k} '}^{2}=-d$, so $d|{x_{k}'}^{2}$,
implying that $\left( \core^{2}(d) \left( d/\core(d) \right) \right)| {x_{k}'}^{2}$. I.e.,
$x_{k}'= \core(d) \sqrt{d/\core(d)} \, x_{k}''$, recalling from the definition of
$\core(d)$, that $d/\core(d)$ is a square. Also if $y_{k}$ is a square and
$y_{k}' \in \bbZ$, then $y_{k}'=y_{k}''$ for $y_{k}'' \in \bbZ$. Hence,
$\core(d) {x_{k}''}^{2}-{ y_{k}''}^{4}=-1$. By the work of Ljunggren \cite{Lj1,Lj2},
there are at most $2$ solutions in positive integers of
$\core(d) x^{2}-y^{4}=-1$. Therefore, there are at most $3$ distinct squares in the
sequence, $\left( y_{k} \right)$, and an even stronger version of
Theorem~\ref{thm:1.2-seq-new} holds for such sequences.
In fact, we have shown
that a stronger version of Conjecture~\ref{conj:1-seq} holds in this case.

\section{Data Availability}

No data beyond that present in the paper itself is used in this work.

\section{Conflict of Interests}

None.

\end{document}